\documentclass{amsart}
\usepackage[utf8]{inputenc}
\usepackage[english]{babel}
\usepackage{amssymb}
\usepackage{amsmath}
\usepackage{amsthm}
\usepackage{tikz}
\usepackage{hyperref}
\usepackage{enumitem}
\usepackage{mathabx}

\DeclareMathOperator{\R}{\mathbb{R}}
\DeclareMathOperator{\N}{\mathbb{N}}
\DeclareMathOperator{\F}{\mathcal{F}}
\newcommand{\ba}{{ba(\widetilde{M})}}

\DeclareMathOperator{\Lip}{Lip}

\DeclareMathOperator{\supp}{supp}

\theoremstyle{plain}
    
    \newtheorem{theorem}{Theorem}[section]
    \newtheorem{proposition}[theorem]{Proposition}
    \newtheorem{lemma}[theorem]{Lemma}
    \newtheorem{corollary}[theorem]{Corollary}

\theoremstyle{definition}
    \newtheorem{example}[theorem]{Example}
    
    \newtheorem{definition}[theorem]{Definition}

\theoremstyle{remark}
    \newtheorem*{remark}{Remark}
    
\numberwithin{equation}{section}

\title[LD$2$P and D$2$P in spaces of Lipschitz functions]{The local diameter two property and the diameter two property in spaces of Lipschitz functions}
\author{Rainis Haller}
\address{{\rm(R. Haller)} Institute of Mathematics and Statistics, University of Tartu, Narva mnt 18, 51009, Tartu, Estonia}%
\email{rainis.haller@ut.ee}

\author{Jaan Kristjan Kaasik}
\address{{\rm(J. K. Kaasik)} Institute of Mathematics and Statistics, University of Tartu, Narva mnt 18, 51009, Tartu, Estonia}%
\email{jaan.kristjan.kaasik@ut.ee}

\author{Andre Ostrak}
\address{{\rm(A. Ostrak)} Department of Mathematics, University of  Agder, Postbox 422 4604 Kristiansand, Norway.}%
\email{andre.ostrak@uia.no}%

\subjclass[2020]{Primary 46B04; Secondary 46B20}%

\keywords{Spaces of Lipschitz functions, Lipschitz-free spaces, diameter two properties, de Leeuw's map, finitely additive measures}%

\begin{document}

\begin{abstract}
We separate the local diameter two property from the diameter two property and their weak-star counterparts from each other in spaces of Lipschitz functions. We characterise the $w^*$-LD$2$P, the $w^*$-D$2$P, the LD$2$P, and the SD$2$P in these spaces. 
We introduce a generalised version of cyclical monotonicity to study functionals on spaces of Lipschitz functions.

\end{abstract}
\maketitle
\section{Introduction}
Throughout the paper, let $M$ be a pointed metric space, i.e. a metric space with a fixed point $0$, with a distance function $d$. We denote by $\Lip_0(M)$ the Banach space of all Lipschitz functions $f\colon M\rightarrow\mathbb{R}$ satisfying $f(0)=0$ equipped with the norm
\[
    \|f\|=\sup\Big\{\frac{|f(x)-f(y)|}{d(x,y)}\colon x, y\in M,\ x\neq y\Big\}.
\]
Let $\delta\colon M\rightarrow \Lip_0(M)^*$ be the canonical isometric embedding of $M$ into $\Lip_0(M)^*$, which is given by $x\mapsto \delta_x$, where $\delta_x(f)=f(x)$. 
The norm closed linear span of $\delta(M)$ in $\Lip_0(M)^*$ is called the \emph{Lipschitz-free space} over $M$ and is denoted by $\mathcal{F}(M)$. It is well known that
$\mathcal{F}(M)^*= \Lip_0(M)$.

Let $X$ be a real Banach space with the closed unit ball $B_X$. Recall \cite{MR3098474, MR4023351} that $X$ has the
\begin{enumerate}
\label{def: d2p-s}
\item \emph{local diameter $2$ property} (briefly, \emph{LD$2$P}) if every slice of $B_X$ has dia\-meter $2$;
\item \emph{diameter $2$ property} (briefly, \emph{D$2$P}) if every nonempty relatively weakly open subset
of $B_X$ has diameter $2$;
\item \emph{strong diameter $2$ property} (briefly, \emph{SD$2$P}) 
if every convex combination of slices of $B_X$ has diameter $2$, 
i.e. the diameter of $\sum_{i=1}^n \lambda_i S_i$ is $2$ 
whenever $n\in\mathbb N$, $\lambda_1,\dotsc,\lambda_n\geq 0$ with $\sum_{i=1}^n\lambda_i=1$, 
and $S_1,\dotsc,S_n$ are slices of $B_X$;
\item\emph{symmetric strong diameter $2$ property} (briefly, \emph{SSD$2$P}) if 
for every $n\in\mathbb N$, every family $\{S_1,\dotsc, S_n\}$ of slices of $B_X$, and every $\varepsilon>0$, 
there exist $f_1\in S_1,\dotsc, f_n\in S_n$, and $g\in B_{X}$ with $\|g\|>1-\varepsilon$ 
such that $f_i\pm g\in S_i$ for every $i\in \{1,\ldots,n\}$.
\end{enumerate}
If $X$ is a dual space, then the weak-star counterparts of these diameter two properties ($w^
*$-LD$2$P, $w^*$-\,D$2$P, $w^*$-\,SD$2$P, and $w^*$-\,SSD$2$P) are defined by replacing slices and weakly open subsets in the above definitions by $w^*$-slices and $w^*$-open subsets, respectively. 
It is known that these properties are distinct and $(4)\Rightarrow (3) \Rightarrow (2) \Rightarrow (1)$, and that these properties are distinct from their weak-star counterparts in dual Banach spaces.

 In this paper, we study these diameter two properties in spaces of Lipschitz functions. The $w^*$-SD$2$P and the $w^*$-SSD$2$P in $\Lip_0(M)$ have been characterised and shown to be different in \cite{MR4026495, MR3803112}. It has been shown in \cite{MR4849357, MR4757422}, that the D$2$P, the SD$2$P, and the SSD$2$P are all different and that all of the diameter two properties are different from their weak-star counterparts in $\Lip_0(M)$ (in fact, it is known that the $w^*$-SSD$2$P does not imply the LD$2$P in these spaces). It has remained unknown whether the ($w^*$-)LD$2$P differs from the ($w^*$-)D$2$P in $\Lip_0(M)$ (see \cite[Question~5.4]{MR4849357} and \cite[p. 3]{MR4757422}). As far as we are aware, there are no known examples of dual Banach spaces with the LD$2$P but without the D$2$P (see \cite{MR4477423, MR3334951, MR4735972} for known examples of Banach spaces with the LD$2$P but without the D$2$P).

 In Example~\ref{Ex:new}, we construct a metric space $M$ for which the space $\Lip_0(M)$ has the LD$2$P, but lacks the $w^*$-D$2$P. Therefore, we answer \cite[Question~5.4]{MR4849357} (see also \cite[p. 3]{MR4757422}) and give an example of a dual Banach space with the LD$2$P that lacks the D$2$P. To achieve this, we introduce and make use of a representation theorem for $\Lip_0(M)^*$, which might be of independent interest. We also characterise the $w^*$-LD$2$P, the $w^*$-D$2$P, the LD$2$P, and the SD$2$P in $\Lip_0(M)$. Recall that the LD$2$P, the D$2$P, and the SD$2$P have dual formulations in terms of local octahedrality, weak octahedrality, and octahedrality (see \cite{MR3150166} and \cite{MR3346197}). Consequently, our results can also be restated in terms of octahedrality.

The paper is outlined as follows. In section 2, we introduce a representation theorem for the dual of the space of Lipschitz functions. In section 3, we present new conditions for the space $\Lip_0(M)$ to have the LD$2$P and the SD$2$P. In section 4, we give characterisations of the $w^*$-LD$2$P and the $w^*$-D$2$P for the spaces of Lipschitz functions. Finally, in section 5, we provide an example of a metric space $M$, for which the space $\Lip_0(M)$ has the LD$2$P but lacks the $w^*$-D$2$P.

In order to study the space $\Lip_0(M)^*$, we make use of the following idea.
Set
\begin{equation*}\label{eq: Gamma:=(M times M) setminus ((x,x): x in M)}
    \widetilde{M} = \big\{ (x,y) \colon x,y\in M,\ x\neq y\big\}.
\end{equation*}
The de Leeuw's map $\Phi\colon \Lip_0(M)\to \ell_{\infty}(\widetilde{M})$ (see \cite[Definition 2.31 and Theorem 2.35]{MR3792558}) is a linear isometry defined by $\Phi f = \tilde{f}$, where
\begin{align*}
    \tilde{f}(x,y)= \frac{f(x)-f(y)}{d(x,y)}.
\end{align*}
Recall that $\ell_\infty(\widetilde M)^*=\ba$, the Banach space of bounded and finitely additive signed measures on the power set of $\widetilde{M}$, with the norm $\|\mu\|=|\mu|(\widetilde{M})$. Thus, for every $F\in \Lip_0(M)^*$, there exists $\mu\in \ba$ with $\|\mu\|=\|F\|$ such that $\Phi^*\mu = F$, i.e.
\begin{equation*}\label{eq: Lip0 - ba(M)}
    F(f)= \int_{\widetilde{M}} \tilde{f} d\mu \qquad \text{for all $f\in \Lip_0(M)$.}
\end{equation*}
One may also regard the de Leeuw's map as embedding $\Lip_0(M)$ into the Banach space $C(\beta \widetilde M)$ of continuous functions on the Stone--\v Cech compactification of $\widetilde M$ (see \cite[p.~102]{MR3792558}). This approach has been used to associate elements of the dual space $\Lip_0(M)^*$ with countably additive Radon measures on $\beta\widetilde{M}$. See, e.g. \cite{MR4771935, aliaga2024leeuwrepresentationsfunctionalslipschitz, aliaga2024solutionextremepointproblem, smith2024choquettheorylipschitzfreespaces} for recent developments using this method.

\textbf{Notations.}
Given a real Banach space $X$, we denote the closed unit ball, the unit sphere, and the dual space of $X$ by $B_X$, $S_X$, and $X^*$, respectively. 
A \emph{slice} of $B_{X}$ is a set of the form
\[
S(x^*, \alpha)=\big\{x\in B_{X}\colon x^*(x)>1-\alpha\big\},
\]
where $x^*\in S_{X^*}$ and $\alpha>0$. 
If $X$ is a dual space, then slices whose defining functional comes from the predual of $X$ are called \emph{$w^*$-slices}.

For $x,y\in M$ with $x\neq y$, 
we define a \emph{molecule} $m_{x,y}$ as a norm one element $\tfrac{\delta_x-\delta_y}{d(x,y)}$ in $\mathcal{F}(M)$. 

Define $\pi\colon\widetilde{M}\rightarrow \mathcal{P}(M)$ by $\pi(x,y)=\{x,y\}$ for all $(x,y)\in \widetilde{M}$. Note, that for a subset $A$ of $\widetilde{M}$, we have
\[
\pi(A)=\big\{x\in M \colon \text{there exists $y\in M$ such that $(x,y)\in A$ or $(y,x)\in A$} \big\}.
\]

\section{Dual of the space of Lipschitz functions}

Every functional from the space $\Lip_0(M)^*$ can be extended norm-preservingly to a finitely additive measure $\mu\in ba(\widetilde{M})$. However, given a measure $\mu \in ba(\widetilde{M})$, it is difficult to determine whether $\|\Phi^*\mu\|=\|\mu\|$. In this section, we provide a way of doing so.

Recall that a subset $A$ of $\widetilde{M}$ is said to be \emph{cyclically monotonic} (see, e.g. \cite{MR4095811, MR2459454}) if for any finite sequence of pairs $(x_1,y_1),\dotsc, (x_n,y_n)\in A$, we have
    \begin{equation*}
        \sum_{i=1}^n d(x_i,y_{i+1}) \geq \sum_{i=1}^n d(x_i,y_i),
    \end{equation*}
   where $y_{n+1}=y_1$.
We introduce the following generalisation of this notion.

\begin{definition}
    Let $\gamma\in(0,1]$ and let $A$ be a subset of $\widetilde{M}$. We say that $A$ is \emph{$\gamma$-cyclically monotonic} if for any finite sequence of pairs $(x_1,y_1),\dotsc, (x_n,y_n)\in A$, we have
    \begin{equation*}
        \sum_{i=1}^n\min \big\{d(x_i,y_{i+1})-\gamma d(x_i,y_i), d(y_i,y_{i+1})\big\}\geq 0,
    \end{equation*}
   where $y_{n+1}=y_1$. 
\end{definition}
\noindent
Note that a subset $A$ of $\widetilde{M}$ is cyclically monotonic if and only if it is $1$-cyclically monotonic. The following result generalises \cite[Theorem~2.4 (i)$\Leftrightarrow$(iv)]{MR4095811}.

\begin{proposition}\label{prop:delta_CM_characterisation}
    Let $A$ be a subset of $\widetilde{M}$ and let $\gamma\in(0,1]$. Then $A$ is $\gamma$-cyclically monotonic if and only if there exists $f\in B_{\Lip_0(M)}$ satisfying $f(m_{x,y})\geq \gamma$ for all $(x,y)\in A$.
\end{proposition}
\begin{proof}
First, assume that $A$ is $\gamma$-cyclically monotonic.
Write $A=\{(x_i,y_i)\in \widetilde{M}\colon i\in I\}$ for a set of indices $I$.
For all $i,j\in I$, define 
    \[
    \beta_{ij}=\min\big\{d(x_i,y_j)-\gamma d(x_i,y_i), d(y_i,y_j)\big\}.
    \]
    It is clear that $\beta_{jj}=0$ for all $j\in I$. Therefore, by \cite[Lemma~2.2]{MR4095811}, there exist real numbers $\alpha_j$, $j\in I$, such that for all $i,j\in I$,
    \[
    \alpha_i\leq \alpha_j+\beta_{ij}.
    \]
    (Note that \cite[Lemma~2.2]{MR4095811} is originally stated for the case $I=\mathbb{N}$, but the result remains valid for any set $I$ with an almost identical proof.)
    
    Define $f\colon M\to \mathbb{R}$ by
    \[
    f(x)=\inf_{i\in I} \bigl(\alpha_i+d(x,y_i)\bigr).
    \]
    Note that for all $i,j\in I$,
    \[
    f(y_i)-f(y_j)=\alpha_i-\alpha_j\leq \beta_{ij} \leq d(y_i,y_j).
    \]
    Thus, $f$ is a $1$-Lipschitz McShane extension from $\{y_i\colon i\in I\}$ to the entire space $M$.
    
    For all $i,j\in I$, we have $\alpha_i\leq \alpha_j+ \beta_{ij}$, and therefore,
    \[
    \alpha_i+\gamma d(x_i,y_i)\leq \alpha_j+d(x_i,y_j).
    \]
    Thus, for all $i\in I$, we have
    \[
    f(x_i)= \inf_{j\in I} \bigl(\alpha_j+d(x_i,y_j)\bigr)\geq  \alpha_i+\gamma d(x_i,y_i)= f(y_i) +  \gamma d(x_i,y_i).
    \]
    Consequently, $f(m_{x,y})\geq \gamma$ for all $(x,y)\in A$.

    Now assume that there exists $f\in B_{\Lip_0(M)}$ satisfying $f(m_{x,y})\geq \gamma$ for all $(x,y)\in A$.
     For any finite sequence of pairs $(x_1,y_1),\dotsc, (x_n,y_n)\in A$ and $y_{n+1}=y_1$, we get
    \begin{align*}
        &\sum_{i=1}^n\min\big\{d(x_i,y_{i+1})-\gamma d(x_i,y_i), d(y_i,y_{i+1})\big\}\\
        &\qquad \geq \sum_{i=1}^n\min\big\{f(x_i)-f(y_{i+1})-f(x_i)+f(y_i), f(y_i)-f(y_{i+1})\big\}\\
        &\qquad =\sum_{i=1}^n f(y_i)-f(y_{i+1})=0.
    \end{align*}
\end{proof}

We are now prepared to provide criteria for determining whether a given $\mu \in \ba$ satisfies $\|\Phi^* \mu\| = \|\mu\|$. Let $\mathfrak{r}\colon \widetilde{M}\to \widetilde{M}$ be the reflection mapping defined by $\mathfrak{r}(x,y)=(y,x)$.

\begin{theorem}\label{thm: Lip_0(M)* description in ba space}
    Let $\mu=\mu^+-\mu^-\in \ba$ be the Jordan decomposition of $\mu$. The following statements are equivalent:
    \begin{enumerate}[label=\textup{(\roman*)}]
        \item $\|\Phi^* \mu\|=\|\mu\|$.
        \item For every $\gamma\in (0,1)$, there exists a subset $A$ of $\widetilde{M}$ and an $f\in B_{\Lip_0(M)}$ satisfying
        \begin{equation*}\label{ineq:ABbigmeasure}
        \mu^+(A)+\mu^-(\mathfrak{r}(A))\geq\gamma |\mu|(\widetilde{M})
        \end{equation*}
        and $f(m_{x,y})\geq \gamma$ for all $(x,y)\in A$.
        \item For every $\gamma\in (0,1)$, there exists a $\gamma$-cyclically monotonic subset $A$ of $\widetilde{M}$ satisfying
        \begin{equation*}
            \mu^+(A)+\mu^-(\mathfrak{r}(A))\geq\gamma |\mu|(\widetilde{M}).
        \end{equation*}
    \end{enumerate}
\end{theorem}

\begin{proof}
    Throughout the proof, we assume without loss of generality that $\|\mu\|=1$. Let $F=\Phi^* \mu$.

    (i)$\Rightarrow$(ii). Let $\gamma\in (0,1)$, let $f\in B_{\Lip_0(M)}$ be such that $F(f)>1- (1-\gamma)^2$, and let 
    \[
    A=\bigl\{(x,y)\in \widetilde{M}\colon f(m_{x,y})\geq \gamma\bigr\}.
    \]
    Then $\mu^+(A)+\mu^-(\mathfrak{r}(A))\geq \gamma$ because, by setting $B=\mathfrak{r}(A)$, we get
    \begin{align*}
        1- (1-\gamma)^2&< F(f)= \int_A \tilde{f} d\mu+\int_B \tilde{f} d\mu +\int_{\widetilde{M}\setminus (A\cup B)} \tilde{f} d\mu\\
        &\leq \mu^+(A)+\mu^-(B)+\gamma |\mu|\big(\widetilde{M}\setminus (A\cup B)\big)\\
        &\leq \mu^+(A)+\mu^-(B)+\gamma \big(1-\mu^+(A)-\mu^-(B)\big)\\
        &=  \gamma+(1-\gamma)\big(\mu^+(A)+\mu^-(B)\big).
    \end{align*}

    (ii)$\Leftrightarrow$(iii) follows directly from Proposition \ref{prop:delta_CM_characterisation}.

    (ii)$\Rightarrow$(i). It suffices to show that $\|F\|\geq 1$. Fix $\gamma\in (0,1)$, let a subset $A$ of $\widetilde{M}$ and a function $f\in B_{\Lip_0(M)}$ be as in (ii), and set $B=\mathfrak{r}(A)$. Then
    \begin{align*}
        \int_{A}\tilde{f}d\mu\geq \gamma\mu^+(A)\quad \text{and}\quad \int_{B}\tilde{f}d\mu \geq \gamma \mu^-(B).
    \end{align*}
    Consequently,
    \begin{align*}
        \|F\|&\geq F(f)=\int_{A}\tilde{f}d\mu+\int_{B}\tilde{f}d\mu+\int_{\widetilde{M}\setminus (A\cup B)}\tilde{f}d\mu\\
        &\geq \gamma \big(\mu^+(A)+\mu^-(B)\big)-(1-\gamma) \\
        &\geq \gamma^2|\mu|(\widetilde{M})-(1-\gamma)\\
        &= \gamma^2-(1-\gamma).
    \end{align*}
    Letting $\gamma \to 1$, we get $\|F\|\geq 1$.
\end{proof}

Given $F\in \Lip_0(M)^*$, the measure $\mu\in \ba$ satisfying $\Phi^* \mu = F$ and $\|\mu\|=\|F\|$ is not unique. We note, similarly to \cite[Proposition~3]{MR4369296}, that $\mu$ can be chosen to be positive.
\begin{proposition}\label{prop: Lip_0(M)* positive measures}
        Let $\nu \in \ba$. Then there exists a positive $\mu\in \ba$ satisfying $\Phi^* \mu = \Phi^* \nu$ and $\|\mu\|=\|\nu\|$. Consequently, for every $F\in \Lip_0(M)^*$, there exists a positive $\mu\in \ba$ satisfying $\Phi^*\mu = F$ and $\|\mu\|=\|F\|$.
\end{proposition}

\begin{proof}
Let $\nu = \nu^+-\nu^-\in \ba$ be the Jordan decomposition of $\nu$ and let $A,B$ be subsets of $\widetilde{M}$ satisfying $A\cap B = \varnothing$, $A\cup B = \widetilde{M}$, and $\nu^+(B)=0$ and $\nu^-(A)=0$. Define $\mu\in \ba$ by
\[
\mu(C) = \nu(A\cap C)-\nu(B\cap \mathfrak{r}(C)).
\]
For any $f\in \Lip_0(M)$, we have
\begin{align*}
    (\Phi^* \mu) (f) &= \int_{\widetilde{M}}\tilde{f}d\mu = \int_A \tilde{f}d\mu+\int_{\mathfrak{r}(B)} \tilde{f}d\mu \\
    &=  \int_A \tilde{f}d\nu+\int_B \tilde{f}d\nu = \int_{\widetilde{M}}\tilde{f}d\nu=( \Phi^* \nu )(f).
\end{align*}
Furthermore,
\begin{align*}
    \|\mu\|=\mu(\widetilde{M})=\nu(A)-\nu(B)=\nu^+(A)+\nu^-(B) = |\nu|(\widetilde{M})=\|\nu\|.
\end{align*}
\end{proof}

\begin{definition}
We call $\mu\in \ba$ \emph{optimal} if it is positive and satisfies $\|\Phi^* \mu\|=\|\mu\|$.
\end{definition}
\noindent
Proposition~\ref{prop: Lip_0(M)* positive measures} thus establishes that for every $F\in \Lip_0(M)^*$, there exists an optimal $\mu\in \ba$ satisfying $\Phi^*\mu = F$.

Throughout the rest of the paper we work exclusively with positive measures. To this end, we state the following special case of Theorem~\ref{thm: Lip_0(M)* description in ba space}.

\begin{corollary}\label{cor: Lip_0(M)* characterisation}
    Let $\mu\in \ba$ be positive. The following statements are equivalent:
    \begin{enumerate}[label=\textup{(\roman*)}]
        \item $\|\Phi^* \mu\|= \|\mu\|$.
        \item For every $\gamma\in (0,1)$, there exist a subset $A$ of $\widetilde{M}$ and an $f\in B_{\Lip_0(M)}$ satisfying $\mu(A)\geq\gamma \mu(\widetilde{M})$ and $f(m_{x,y})\geq \gamma$ for all $(x,y)\in A$.
        \item For every $\gamma\in (0,1)$, there exists a $\gamma$-cyclically monotonic subset $A$ of $\widetilde{M}$ satisfying $\mu(A)\geq\gamma \mu(\widetilde{M})$.
    \end{enumerate}
\end{corollary}

We conclude this section by showing that, for certain metric spaces, a strengthened version of Corollary~\ref{cor: Lip_0(M)* characterisation} holds. To this end, we rely on the following observation.

\begin{lemma}\label{lemma: Integer metric space cyclically monotonic}
    Let $n\in \mathbb{N}$ and let $M$ be a metric space with $d(x,y)\in \{0,1,\dotsc,n\}$ for all $x,y\in M$. Let $\mu\in \ba$ be positive, let $\gamma\in (0,1)$, and let $A$ be a $\gamma$-cyclically monotonic subset of $\widetilde{M}$. Then there exists a cyclically monotonic subset $B$ of $A$ with $\mu(B)\geq  \mu(A)-2n(1-\gamma)\mu(\widetilde{M})$.
\end{lemma}

\begin{proof}
    We assume without loss of generality, that $\mu(\widetilde{M})=1$ and $n(1-\gamma)< 1$. Throughout the proof,
    we denote the floor, ceiling, and fractional part of $\alpha\in\mathbb{R}$ by $\lfloor \alpha \rfloor$, $\lceil \alpha\rceil$, and $\{\alpha\}$, respectively. Write $A=\{(x_i,y_i)\in \widetilde{M}\colon i\in I\}$ for a set of indices $I$. For all $i,j \in I$, we define 
    \[
        \beta_{ij}=d(x_i,y_j)-d(x_i,y_i)
    \]
    and
    \[
        \widetilde{\beta}_{ij}=\min\bigl\{d(x_i,y_j)-\gamma d(x_i,y_i),d(y_i,y_j)\bigr\},
    \]
    and note that $\widetilde{\beta}_{ij}\leq\beta_{ij}+n(1-\gamma)$.
    
    By \cite[Lemma 2.2]{MR4095811}, there exist $\widetilde{\alpha}_i\in \mathbb{R}$, $i\in I$, so that $\widetilde{\alpha}_i\leq\widetilde{\alpha}_j+ \widetilde{\beta}_{ij}$ (as in the proof of Proposition~\ref{prop:delta_CM_characterisation}, we note that \cite[Lemma 2.2]{MR4095811} remains valid for any set $I$). By \cite[Lemma 2.2]{MR4095811}, it suffices to show that there exist a subset $J$ of $I$ and elements $\alpha_i\in \mathbb{R}$, $i\in J$, so that $\alpha_i\leq \alpha_j+\beta_{ij}$ for all $i,j\in J$ and $\mu(B)\geq \mu(A)-2n(1-\gamma)$, where $B=\{(x_i,y_i)\colon i\in J\}$.

    Fix $K\in \mathbb{N}$ so that $\frac{1}{2K}\leq n(1-\gamma)\leq\frac1K$. For every $k\in\{1,\ldots,K\}$, define a subset $I_k$ of $I$ by
    \[
    I_k=\Bigl\{i\in I\colon \{\widetilde{\alpha}_i\} \in\Big[\frac{k-1}{K},\frac{k}{K}\Big)\Bigr\}
    \] 
    and a subset $A_k=\{(x_i,y_i)\colon i\in I_k\}$ of $A$. Fix $k\in\{1,\ldots,K\}$ such that $\mu(A_k)\leq 1/K$. 
    Let $J=I\setminus I_k$, let $B=\{(x_i,y_i)\colon i\in J\}$, and let $\alpha_i=\lfloor\widetilde{\alpha}_i-\frac{k}{K}\rfloor$ for each $i\in J$. Note that $B=A\setminus A_k$ and that
    \[
    \mu(B)= \mu(A)-\mu(A_k)\geq \mu(A) -\frac{1}{K}\geq \mu(A)-2n(1-\gamma).
    \]    
    
    Fix $i,j\in J$. It remains to verify that $\alpha_i\leq \alpha_j+\beta_{ij}$. Note that
    \[
    \alpha_i-\alpha_j=\begin{cases}
        \lfloor \widetilde{\alpha}_i-\widetilde{\alpha}_j\rfloor, 
        &\text{if $\{\widetilde{\alpha}_i-\frac{k}{K}\}\geq\{\widetilde{\alpha}_j-\frac{k}{K}\}$};\\
        \lfloor \widetilde{\alpha}_i-\widetilde{\alpha}_j\rfloor+1, 
        &\text{if $\{\widetilde{\alpha}_i-\frac{k}{K}\}<\{\widetilde{\alpha}_j-\frac{k}{K}\}$}.
    \end{cases}
    \]
    If $\{\widetilde{\alpha}_i-\frac{k}{K}\}\geq\{\widetilde{\alpha}_j-\frac{k}{K}\}$, then
    \[
    \alpha_i-\alpha_j=\lfloor \widetilde{\alpha}_i-\widetilde{\alpha}_j\rfloor\leq \lfloor\widetilde{\beta}_{ij}\rfloor\leq \lfloor\beta_{ij}+n(1-\gamma) \rfloor=\beta_{ij}.
    \]
    
    Now assume that $\{\widetilde{\alpha}_i-\frac{k}{K}\}<\{\widetilde{\alpha}_j-\frac{k}{K}\}$. It suffices to show that $\widetilde{\alpha}_i-\widetilde{\alpha}_j<\beta_{ij}$ because then 
    \[
    \alpha_i-\alpha_j=\lfloor \widetilde{\alpha}_i-\widetilde{\alpha}_j\rfloor+1\leq \lceil \beta_{ij}\rceil = \beta_{ij}.
    \]
    To obtain a contradiction, suppose that $\widetilde{\alpha}_i-\widetilde{\alpha}_j\geq\beta_{ij}$. Then
    \[
    \beta_{ij}\leq \widetilde{\alpha}_i-\widetilde{\alpha}_j\leq \widetilde{\beta}_{ij} \leq \beta_{ij}+n(1-\gamma),
    \]
    and therefore,
    \[
    \{\widetilde{\alpha}_i-\frac{k}{K}\}-\{\widetilde{\alpha}_j-\frac{k}{K}\}=\{\widetilde{\alpha}_i-\widetilde{\alpha}_j\}-1\leq n(1-\gamma)-1.
    \]
    Consequently,
    \[
    \{\widetilde{\alpha}_j-\frac{k}{K}\}\geq 1-n(1-\gamma)+\{\widetilde{\alpha}_i-\frac{k}{K}\}\geq 1-\frac{1}{K}.
    \]
    However, by the definition of $J$, we have $\{\widetilde{\alpha}_j-\frac{k}{K}\}\in[0,1-\frac{1}{K})$.
\end{proof}

\begin{proposition}\label{prop: optimal mu Z}
 Let $n\in \mathbb{N}$ and let $M$ be a pointed metric space with $d(x,y)\in \{0,1,\dotsc,n\}$ for all $x,y\in M$. Let $\mu\in\ba$ be positive. The following conditions are equivalent:

\begin{enumerate}[label=\textup{(\roman*)}]
    \item $\|\Phi^*\mu\|=\|\mu\|$.
    \item For every $\gamma\in (0,1)$, there exist a subset $B$ of $\widetilde{M}$ and an $f\in B_{\Lip_0(M)}$ satisfying $\mu(B)\geq\gamma \mu(\widetilde{M})$ and $f(m_{x,y})=1$ for all $(x,y)\in B$.
    \item For every $\gamma\in(0,1)$, there exists a cyclically monotonic set $B$ of $\widetilde{M}$ with $\mu(B)\geq \gamma\mu(\widetilde{M})$.
\end{enumerate}
\end{proposition}
\begin{proof}
    (ii)$\Leftrightarrow$(iii) follows directly from Proposition~\ref{prop:delta_CM_characterisation}.

    (iii)$\Rightarrow$(i) is clear from Corollary~\ref{cor: Lip_0(M)* characterisation} (iii)$\Rightarrow$(i).
    
    (i)$\Rightarrow$(iii). Let $\gamma\in(0,1)$ and let $\widetilde{\gamma}=\frac{2n+\gamma}{2n+1}$. By Corollary~\ref{cor: Lip_0(M)* characterisation} (i)$\Rightarrow$(iii), there exists a $\widetilde{\gamma}$-cyclically monotonic subset $A$ of $\widetilde{M}$ with $\mu(A)\geq \widetilde{\gamma} \mu(\widetilde{M})$. By Lemma~\ref{lemma: Integer metric space cyclically monotonic}, there exists a cyclically monotonic subset $B$ of $A$ satisfying
    \begin{align*}
    \mu(B)&\geq \mu(A)-2n(1-\widetilde{\gamma})\mu(\widetilde{M}) \geq \big(\widetilde{\gamma}-2n(1-\widetilde{\gamma})\big)\mu(\widetilde{M}) \\
    &= \big((2n+1)\widetilde{\gamma}-2n\big)\mu(\widetilde{M})=\gamma\mu(\widetilde{M}).
    \end{align*}
    
\end{proof}

\section{\texorpdfstring{Characterisations of LD$2$P and SD$2$P for spaces of Lipschitz functions.}{Characterisation of diameter 2 properties in Lip0(M)}}

In this section we characterise the LD$2$P and the SD$2$P for the spaces of Lipschitz functions. The following lemma is useful for working with slices of such Banach spaces.

\begin{lemma}\label{lem:f in slice implies big A}
Let $\mu\in S_\ba$ be optimal, let $\alpha>0$, let $f\in S(\Phi^*\mu,\alpha^2)$, and let 
    \[
    A=\bigl\{(x,y)\in \widetilde{M}\colon f(m_{x,y})\geq 1-\alpha\bigr\}.
    \]
    Then $\mu(A)\geq 1-\alpha$.  
\end{lemma}

\begin{proof}
It suffices to note that
    \begin{align*}
     1-\alpha^2\leq (\Phi^*\mu)(f) &= \int_A \tilde{f} d\mu +\int_{\widetilde{M}\setminus A} \tilde{f} d\mu\leq \mu(A)+(1-\alpha)\mu(\widetilde{M}\setminus A) \\
     &= \mu(A)+(1-\alpha)\big(1-\mu(A)\big)=1-\alpha+\alpha\mu(A).
    \end{align*}
\end{proof}

We are now ready to characterise the LD$2$P for the spaces of Lipschitz functions.

\begin{proposition}\label{prop: LD2P delta-CM characterisation}
    The space $\Lip_0(M)$ has the LD$2$P if and only if for all optimal $\mu\in S_\ba$ and $\gamma\in (0,1)$, there exist a subset $A$ of $\widetilde{M}$ and $u,v\in M$ with $u\neq v$ such that $\mu(A)\geq \gamma$ and the sets $A\cup \{(u,v)\}$ and $A\cup \{(v,u)\}$ are both $\gamma$-cyclically monotonic.
\end{proposition}

\begin{proof}
   ($\Leftarrow$) Let $F\in S_{\Lip_0(M)^*}$ and let $\alpha>0$. It suffices to show that for any $\gamma\in (0,1)$, there exist $f,g\in S(F,\alpha)$ satisfying $\|f-g\|\geq 2\gamma$. Let $\gamma\in (0,1)$ be such that $\gamma^2 \geq 1-\frac{\alpha}{2}$. By Proposition~\ref{prop: Lip_0(M)* positive measures}, there exists an optimal $\mu\in S_\ba$ with $\Phi^* \mu = F$. Therefore, there exist a subset $A$ of $\widetilde{M}$ and $u,v\in M$ with $u\neq v$ so that $\mu(A)\geq \gamma$ and the sets $A\cup \{(u,v)\}$ and $A\cup \{(v,u)\}$ are both $\gamma$-cyclically monotonic.
   By Proposition~\ref{prop:delta_CM_characterisation}, there exist $f,g\in B_{\Lip_0(M)}$ such that $f(m_{x,y})\geq \gamma$ and $g(m_{x,y})\geq \gamma$ for all $(x,y)\in A$, and $f(m_{u,v})\geq \gamma$ and $g(m_{v,u})\geq \gamma$.
    Therefore,     
    \[
        \|f-g\|\geq (f-g)(m_{u,v})\geq 2\gamma.
    \]
    Note that $f\in S(F,\alpha)$ because
    \[
    F(f)=\int_{\widetilde{M}}\tilde{f}d\mu = \int_{A}\tilde{f}d\mu+\int_{\widetilde{M}\setminus A} \tilde{f}d\mu \geq \gamma^2-(1-\gamma)> 2\gamma^2-1 \geq 1-\alpha.
    \]
    Similarly, we get that $g\in S(F,\alpha)$.

    ($\Rightarrow$) Now assume that the space $\Lip_0(M)$ has the LD$2$P. Let $\mu\in S_\ba$ be optimal and let $\gamma\in(0,1)$. Set $\alpha=\frac{1-\gamma}{2}$. Since $\Lip_0(M)$ has the LD$2$P, there exist $f,g\in S(\Phi^* \mu,\alpha^2)$ with $\|f-g\| > 2-\alpha$. Therefore, there exist $u,v\in M$ so that
    \[
    f(m_{u,v})+g(m_{v,u}) = (f-g)(m_{u,v})\geq 2-\alpha.
    \]
    Thus, $f(m_{u,v})\geq 1-\alpha$ and $g(m_{v,u})\geq 1-\alpha$. Set $A = B\cap C$, where
    \[
        B=\big\{(x,y)\in \widetilde{M}\colon f(m_{x,y})\geq  1-\alpha\big\} \quad\text{and}\quad  C=\big\{(x,y)\in \widetilde{M}\colon g(m_{x,y})\geq  1-\alpha\big\}.
    \]
    By Proposition~\ref{prop:delta_CM_characterisation}, $B$ and $C$ are $\gamma$-cyclically monotonic because $\gamma<1-\alpha$. Note that $A\cup \{(u,v)\}$ and $A\cup \{(v,u)\}$ are $\gamma$-cyclically monotonic because they are subsets of $\gamma$-cyclically monotonic sets $B$ and $C$, respectively. By Lemma~\ref{lem:f in slice implies big A}, we have $\mu(B)\geq 1-\alpha$ and $\mu(C)\geq 1-\alpha$, and hence,
    \[
    \mu(A)\geq 1-2\alpha = \gamma.
    \]
    
\end{proof}

The following lemma provides a tool to check whether the set $A\cup \{(u,v)\}$ is $\gamma$-cyclically monotonic, given a $\gamma$-cyclically monotonic subset $A$ of $\widetilde{M}$ and $(u,v)\in \widetilde{M}$.

\begin{lemma}\label{lem: A cup (u,v) delta-CM}
    Let $\gamma\in(0,1]$, let $A$ be a subset of $\widetilde{M}$, and let $u,v\in M$ with $u\neq v$. The set $A\cup \{(u,v)\}$ is $\gamma$-cyclically monotonic if and only if there exists $f\in B_{\Lip_0(M)}$ satisfying for all $(x,y)\in A$, $f(m_{x,y})\geq \gamma$, 
    and for all $x,y\in \pi(A)$,
    \begin{equation*}
         f(y)-f(x)+ \gamma d(u,v)\leq d(x,u)+d(y,v).
    \end{equation*}
\end{lemma}
\begin{proof}
    Assume first that $A\cup \{(u,v)\}$ is $\gamma$-cyclically monotonic. By Proposition~\ref{prop:delta_CM_characterisation}, there exists $f\in B_{\Lip_0(M)}$ satisfying $f(m_{x,y})\geq \gamma$ for all $(x,y)\in A\cup \{(u,v)\}$. For all $x,y\in \pi(A)$, we have
    \[
    \gamma d(u,v)\leq f(u)-f(v)\leq  f(x)+d(x,u)-f(y)+d(y,v).
    \]

    Now assume that there exists $f\in B_{\Lip_0(M)}$ satisfying $f(m_{x,y})\geq \gamma$ for all $(x,y)\in A$, and for all $x,y\in \pi(A)$,
    \begin{equation*}
         f(y)-f(x)+ \gamma d(u,v)\leq d(x,u)+d(y,v).
    \end{equation*}
    We show that the set $A\cup \{(u,v)\}$ is $\gamma$-cyclically monotonic. By Proposition~\ref{prop:delta_CM_characterisation}, it suffices to find $g\in B_{\Lip_0(M)}$ satisfying $g(m_{x,y})\geq \gamma$ for all $(x,y)\in A\cup \{(u,v)\}$.
    Without loss of generality, we may and do assume that $0\in \pi(A)$. Define $g\in B_{\Lip_0(M)}$ by setting $g|_{\pi(A)}=f|_{\pi(A)}$ and $g(u)=\inf_{x\in \pi(A)} (f(x)+d(x,u))$, and then
    \begin{align*}
            \quad g(y)=\sup_{x\in \pi(A)\cup \{u\}} \big(g(x)-d(x,y)\big) \qquad \text{for all $y\in M\setminus \big(\pi(A)\cup \{u\}\big)$.}
    \end{align*}
    Clearly $g(m_{x,y})=f(m_{x,y})\geq \gamma$ for all $(x,y)\in A$. Note that $g(m_{u,v})\geq \gamma$. Indeed, if $g(v)=g(u)-d(u,v)$, then this is clear; otherwise, 
    \[
    g(u)-g(v)=\inf_{x,y\in \pi(A)} \big(f(x)+d(x,u)-f(y)+d(y,v)\big)\geq \gamma d(u,v).
    \]
\end{proof}

From Proposition~\ref{prop: LD2P delta-CM characterisation} and Lemma~\ref{lem: A cup (u,v) delta-CM}, we immediately get another characterisation of the LD$2$P for the spaces of Lipschitz functions.
\begin{proposition}\label{prop: LD2P characterisation}
    The space $\Lip_0(M)$ has the LD$2$P if and only if for all optimal $\mu\in S_\ba$ and $\gamma\in(0,1)$, there exist a subset $A$ of $\widetilde{M}$ with $\mu(A)\geq \gamma$, functionals $f,g\in B_{\Lip_0(M)}$, and elements $u,v\in M$ with $u\neq v$ satisfying for all $(x,y)\in A$, $f(m_{x,y})\geq \gamma$ and $g(m_{x,y})\geq \gamma$, and for all $x,y\in \pi(A)$,
    \begin{equation*}\label{eq: LD2P-fg}
        \max\big\{f(x)-f(y),g(y)-g(x)\big\}+\gamma d(u,v)\leq d(x,u)+d(y,v).
    \end{equation*}
\end{proposition}

We now characterise the SD$2$P for the spaces of Lipschitz functions in the spirit of Proposition~\ref{prop: LD2P delta-CM characterisation}. It was shown in \cite[Example~3.1]{MR4849357} that the $w^*$-SD$2$P and the SD$2$P are not equivalent properties for the spaces of Lipschitz functions. Therefore, our characterisation is not equivalent to the one given for the $w^*$-SD$2$P in \cite[Theorem~3.1]{MR3803112}.

\begin{proposition}\label{prop: SD2P delta-CM characterisation}
    The space $\Lip_0(M)$ has the SD$2$P if and only if for all optimal $\mu_1,\dotsc, \mu_n\in S_\ba$ and $\gamma\in(0,1)$, there exist subsets $A_1,\dotsc, A_n$ of $\widetilde{M}$ and $u,v\in M$ with $u\neq v$ such that for all $i\in \{1,\dotsc, n\}$, $\mu_i(A_i)\geq \gamma$ and the sets $A_i\cup \{(u,v)\}$ and $A_i\cup \{(v,u)\}$ are $\gamma$-cyclically monotonic.
\end{proposition}
\begin{proof}
($\Leftarrow$) Let $F_1,\dotsc, F_n\in S_{\Lip_0(M)^*}$, let $\alpha>0$, let $\lambda_1,\dotsc,\lambda_n\geq 0$ satisfy $\sum_{i=1}^n\lambda_i=1$, and let $C=\sum_{i=1}^n\lambda_i S(F_i,\alpha)$. It suffices to show that for any $\gamma\in (0,1)$, there exist $f,g\in C$ satisfying $\|f-g\|\geq 2\gamma$. Let $\gamma\in(0,1)$ be such that $\gamma^2\geq 1-\frac{\alpha}{2}$. By Proposition~\ref{prop: Lip_0(M)* positive measures}, there exist optimal $\mu_1,\dotsc,\mu_n\in S_\ba$ with $\Phi^* \mu_i = F_i$ for every $i\in\{1,\dotsc,n\}$. Therefore, there exist subsets $A_1,\dotsc,A_n$ of $\widetilde{M}$ and $u,v\in M$ with $u\neq v$ such that for all $i\in \{1, \dotsc, n\}$, $\mu_i(A_i)\geq \gamma$ and the sets $A_i\cup \{(u,v)\}$ and $A_i\cup \{(v,u)\}$ are $\gamma$-cyclically monotonic.

Fix $i\in \{1,\dotsc, n\}$. By Theorem~\ref{thm: Lip_0(M)* description in ba space}, there exist $f_i,g_i\in B_{\Lip_0(M)}$ such that $f_i(m_{x,y})\geq \gamma$ and $g_i(m_{x,y})\geq \gamma$ for all $(x,y)\in A_i$, and $f_i(m_{u,v})\geq \gamma$ and $g_i(m_{v,u})\geq \gamma$. 
    Note that $f_i\in S(F_i,\alpha)$ because
    \[
    F_i(f_i)=\int_{\widetilde{M}}\tilde{f_i}d\mu_i = \int_{A_i}\tilde{f_i}d\mu_i+\int_{\widetilde{M}\setminus A_i} \tilde{f_i}d\mu_i \geq \gamma^2-(1-\gamma)> 2\gamma^2-1 \geq 1-\alpha.
    \]
    Similarly, we get that $g_i\in S(F_i,\alpha)$. 
    
    Define $f=\sum_{i=1}^n \lambda_i f_i$ and $g=\sum_{i=1}^n \lambda_i g_i$. Then $f,g\in C$ and      
    \[
        \|f-g\|\geq \sum_{i=1}^n \lambda_i(f_i-g_i)(m_{u,v})\geq 2\gamma.
    \]
    
($\Rightarrow$)
    Now assume that $\Lip_0(M)$ has the SD$2$P. Let $\mu_1,\dotsc, \mu_{n}\in S_\ba$ be optimal and let $\gamma\in(0,1)$. Set $\alpha=\frac{1-\gamma}{2}$. Since $\Lip_0(M)$ has the SD$2$P, there exist $f_i,g_i\in S(\Phi^* \mu_i,\alpha^2)$, $i\in \{1,\ldots,n\}$, such that
    \[
    \Big\|\frac{1}{n}\sum_{i=1}^n (f_i-g_i)\Big\|>  2-\frac{\alpha}{n}.
    \]
    Let $u,v\in M$ with $u\neq v$ be such that 
    \[
        \frac{1}{n}\sum_{i=1}^n (f_i-g_i)(m_{u,v})\geq 2-\frac{\alpha}{n}.
    \]
    Fix $i\in \{1,\dotsc,n\}$. We have
    \[
    f_i(m_{u,v})+g_i(m_{v,u})=(f_i-g_i)(m_{u,v})\geq n\big(2-\frac{\alpha}{n}\big)-2(n-1)= 2-\alpha,
    \]
    and therefore, $f_i(m_{u,v})\geq 1-\alpha$ and $g_i(m_{v,u})\geq 1-\alpha$. Set $A_i = B_i\cap C_i$, where
    \[
        B_i=\big\{(x,y)\in \widetilde{M}\colon f_i(m_{x,y})\geq  1-\alpha\big\} \quad\text{and}\quad  C_i=\big\{(x,y)\in \widetilde{M}\colon g_i(m_{x,y})\geq  1-\alpha\big\}.
    \]
    By Proposition~\ref{prop:delta_CM_characterisation}, $B_i$ and $C_i$ are $\gamma$-cyclically monotonic because $\gamma<1-\alpha$. Note that $A_i\cup \{(u,v)\}$ and $A_i\cup \{(v,u)\}$ are $\gamma$-cyclically monotonic because they are subsets of $\gamma$-cyclically monotonic sets $B_i$ and $C_i$, respectively.
    By Lemma~\ref{lem:f in slice implies big A}, we have $\mu_i(B_i)\geq 1-\alpha$ and $\mu_i(C_i)\geq 1-\alpha$, and hence,
    \[
    \mu_i(A_i)\geq 1-2\alpha=\gamma.
    \]
\end{proof}

From Lemma~\ref{lem: A cup (u,v) delta-CM} and Proposition~\ref{prop: SD2P delta-CM characterisation}, we get another characterisation of the SD$2$P for the spaces of Lipschitz functions, in the spirit of Proposition~\ref{prop: LD2P characterisation}.

\begin{proposition}\label{prop: SD2P characterisation}
    The space $\Lip_0(M)$ has the SD$2$P if and only if for all optimal $\mu_1,\dotsc,\mu_n\in S_\ba$ and $\gamma\in(0,1)$, there exist subsets $A_1,\dotsc,A_n$ of $\widetilde{M}$ with $\mu_1(A_1)\geq \gamma,\dotsc,\mu_n(A_n)\geq \gamma$, functionals $f_1,g_1,\dotsc,f_n,g_n\in B_{\Lip_0(M)}$, and elements $u,v\in M$ with $u\neq v$ satisfying for all $i\in \{1,\dotsc,n\}$ and $(x,y)\in A_i$, $f_i(m_{x,y})\geq \gamma$ and $g_i(m_{x,y})\geq \gamma$,
    and for all $i\in \{1,\dotsc,n\}$ and $x,y\in \pi(A_i)$,
    \[
        \max\big\{f_i(x)-f_i(y),g_i(y)-g_i(x)\big\}+\gamma d(u,v)\leq d(x,u)+d(y,v).
    \]
\end{proposition}

\begin{remark}
In \cite{MR4849357}, it was shown that if $M$ has a property called the FLTP, then $\Lip_0(M)$ has the SD$2$P. It remains unknown whether the FLTP is equivalent to the condition described in Proposition~\ref{prop: SD2P characterisation}.
\end{remark}

\section{\texorpdfstring{Characterisations of $w^*$-D$2$P and $w^*$-LD$2$P for spaces of Lipschitz functions}{Metric characterisation of w*-D2P and w*-LD2P for spaces of Lipschitz functions}}

Recall that the space $\Lip_0(M)$ has the $w^*$-SD$2$P if and only if the space $M$ has the LTP \cite[Theorem~3.1]{MR3803112}, and that the space $\Lip_0(M)$ has the $w^*$-SSD$2$P if and only if the space $M$ has the SLTP \cite[Theorem~2.1]{MR4026495}. Inspired by these results, we characterise the $w^*$-D$2$P and the $w^*$-LD$2$P for the spaces of Lipschitz functions.

\begin{definition}
    We say that $M$ has the \emph{Lip-LTP} if, given a finite subset $N$ of $M$, $\varepsilon>0$, and $f\in B_{\Lip_0(M)}$, there exist $u,v\in M$ with $u\neq v$ satisfying for all $x,y\in N$,
    \begin{equation}\label{ineq: Lip-LTP}
         (1-\varepsilon)\big(|f(x)-f(y)|+d(u,v)\big)\leq d(x,u)+d(y,v).
    \end{equation}
\end{definition}
\noindent
Clearly the Lip-LTP follows from the LTP.

\begin{proposition}\label{prop: w^*-D2P characterisation}
    The space $\Lip_0(M)$ has the $w^*$-D$2$P if and only if $M$ has the Lip-LTP.
\end{proposition}

\begin{proof}

If $M$ is unbounded or not uniformly discrete, then $M$ has the LTP and $\Lip_0(M)$ has the $w^*$-SD$2$P by \cite[Theorem~2.4]{MR3826487} and \cite[Theorem~3.1]{MR3803112}. Therefore, we assume throughout the proof that $M$ is bounded and uniformly discrete.  i.e. there exist $r,R>0$ such that for all $x,y\in M$ with $x\neq y$ we have $r<d(x,y)<R$.

Assume first that $\Lip_0(M)$ has the $w^*$-D$2$P. 
Let $N$ be a finite subset of $M$, let $\varepsilon>0$, and let $f\in B_{\Lip_0(M)}$. Let $\delta=\frac{\varepsilon r}{2R}$ and define 
    \[
    U = \big\{g\in B_{\Lip_0(M)} \colon \text{$|(f-g)(m_{x,y})|<\delta$ for all $x,y\in N$ with $x\neq y$}\big\}.
    \]
    Note that $U$ is not empty because $f\in U$. Since $\Lip_0(M)$ has the $w^*$-D$2$P, there exist $g,h\in U$ with $\|g-h\|> 2-\varepsilon/2$. Consequently, there exist $u,v\in M$ with $u\neq v$ such that
    \[
    g(m_{u,v})+h(m_{v,u})=(g-h)(m_{u,v})\geq 2-\varepsilon/2.
    \]
    Thus, $g(m_{u,v})\geq 1-\varepsilon/2$ and $h(m_{v,u})\geq 1-\varepsilon/2$. 
    
 Fix $x,y\in N$. If $f(x)-f(y)\geq 0$, then 
    \begin{align*}
        (1-\varepsilon)\big(|f(x)-f(y)|+d(u,v)\big) & \leq  f(x)-f(y)+ (1-\varepsilon)d(u,v) \\
        &\leq h(x)-h(y)+\delta d(x,y) +(1-\varepsilon)d(u,v) \\
        &\leq h(x)-h(y) +\frac{\varepsilon}{2} d(u,v)+(1-\varepsilon)d(u,v) \\
        &\leq h(x)-h(y)+h(v)-h(u)\\
        &\leq d(x,u)+d(y,v).
    \end{align*}
Similarly, if $f(x)-f(y)< 0$, then 
        \begin{align*}
        (1-\varepsilon)\big(|f(x)-f(y)|+d(u,v)\big) &\leq  f(y)-f(x)+ (1-\varepsilon)d(u,v) \\
        &\leq g(y)-g(x)+\delta d(x,y) +(1-\varepsilon)d(u,v) \\
        &\leq g(y)-g(x) +\frac{\varepsilon}{2} d(u,v)+(1-\varepsilon)d(u,v) \\
        &\leq g(y)-g(x)+g(u)-g(v)\\
        &\leq d(x,u)+d(y,v).
    \end{align*}

    Now assume that $M$ has the Lip-LTP. 
    Let $\mu_1,\ldots,\mu_n\in S_{\F(M)}$ be finitely supported and let $\alpha_1,\ldots,\alpha_n>0$. Assume that $U=\bigcap_{i=1}^n S(\mu_i,\alpha_i)$ is not empty and fix $f\in U$. Let $N=\{0\}\cup (\bigcup_{i=1}^n \supp \mu_i)$, let $\delta\in (0,1)$, and let $\varepsilon=\frac{\delta r}{2R}$. Since $M$ has the Lip-LTP, there exist $u,v\in M$ with $u\neq v$ such that the
    inequality (\ref{ineq: Lip-LTP}) holds for all $x,y\in N$. Define a function $g \in B_{\Lip_0(M)}$ by first setting $g|_N=f|_N$ and $g(u)=\min_{x\in N} (f(x)+d(x,u))$, and then
    \begin{align*}
            \quad g(y)=\max_{x\in N\cup \{u\}} \big(g(x)-d(x,y)\big) \qquad \text{for all $y\in M\setminus \big(N\cup \{u\}\big)$.}
    \end{align*}
    Note that $g(u)-g(v)\geq (1-\delta)d(u,v)$. Indeed, if $g(v)=g(u)-d(u,v)$, then this is clear; otherwise, there exist $x,y\in N$ such that
    \begin{align*}
     g(u)-g(v) &= f(x)+d(x,u)-f(y)+d(y,v) \\
     &\geq (1-\varepsilon)d(u,v) + \varepsilon \big(f(x)-f(y)\big)\\
     &\geq (1-\varepsilon)d(u,v) - \varepsilon d(x,y) \\
     &\geq \big(1-\frac{\delta r}{2R}\big)d(u,v) - \frac{\delta}{2} d(u,v) \geq (1-\delta)d(u,v).
    \end{align*}
    Similarly, define a function $h \in B_{\Lip_0(M)}$ by first setting $h|_N=f|_N$ and $h(v)=\min_{x\in N} (f(x)+d(x,v))$, and then
    \begin{align*}
            \quad h(y)=\max_{x\in N\cup \{v\}} \big(h(x)-d(x,y)\big) \qquad \text{for all $y\in M\setminus \big(N\cup \{v\}\big)$,}
    \end{align*}
    and verify that $h(v)-h(u)\geq (1-\delta)d(u,v)$. It remains to note that $g,h\in U$ and that 
    \[
    \|g-h\|\geq (g-h)(m_{u,v})\geq 2-2\delta.
    \]
\end{proof}

In Proposition~\ref{prop: LD2P delta-CM characterisation} and Proposition~\ref{prop: LD2P characterisation}, we gave two characterisations of the LD$2$P for the spaces of Lipschitz functions. We now give similar characterisations of the $w^*$-LD$2$P. Note that the $w^*$-LD$2$P and the LD$2$P are not equivalent properties for the spaces of Lipschitz functions by \cite[Example~3.1]{MR4849357}.

\begin{proposition}\label{prop: w^*-LD2P characterisation with CM}
    The space $\Lip_0(M)$ has the $w^*$-LD$2$P if and only if, given a finite cyclically monotonic subset $A$ of $\widetilde{M}$ and $\gamma\in(0,1)$, there exist $u,v\in M$ with $u\neq v$ such that $A\cup \{(u,v)\}$ and $A\cup \{(v,u)\}$ are $\gamma$-cyclically monotonic.
\end{proposition}

\begin{proof} ($\Rightarrow$)
Assume that $\Lip_0(M)$ has the $w^*$-LD$2$P. Let $A = \{(x_i,y_i)\colon i=1,\dotsc,n\}$ be a finite cyclically monotonic subset of $\widetilde{M}$ and let $\gamma\in(0,1)$. Set $\alpha=\frac{1-\gamma}{n}$ and $\mu=\sum_{i=1}^n \frac{1}{n} m_{x_i,y_i}$. By Proposition~\ref{prop:delta_CM_characterisation}, we have $\mu\in S_{\mathcal{F}(M)}$. Since $\Lip_0(M)$ has the $w^*$-LD$2$P, there exist $f,g\in S(\mu, \alpha)$ satisfying  $\|f-g\|>  1+\gamma$.
Therefore, there exist $u,v\in M$ with $u\neq v$ such that $(f-g)(m_{u,v})\geq 1+\gamma$. Thus, $f(m_{u,v})\geq \gamma$ and $g(m_{v,u})\geq \gamma$. By Proposition~\ref{prop:delta_CM_characterisation}, $A\cup \{(u,v)\}$ is $\gamma$-cyclically monotonic because from $f\in S(\mu,\alpha)$ we get for any $i\in \{1,\dotsc,n\}$,
\begin{align*}
    f(m_{x_i,y_i})&= n f(\mu)- \sum_{\substack{j = 1\\j\neq i}}^n f(m_{x_j,y_j}) \geq n(1-\alpha)-(n-1) = 1-n\alpha= \gamma.
\end{align*}
Analogously, we get that $A\cup \{(v,u)\}$ is $\gamma$-cyclically monotonic because $g(m_{x_i,y_i})\geq \gamma$ for all $i\in \{1,\dotsc,n\}$.

($\Leftarrow$) 
Let $\mu\in S_{\F(M)}$ be finitely supported, let $0<\alpha\leq 1$, and let $\gamma\in(1-\alpha,1)$. It suffices to show that there exist $f,g\in S(\mu,\alpha)$ with $\|f-g\|\geq 2\gamma$. By \cite[Proposition~3.16]{MR3792558}, there exist $n\in \mathbb{N}$, $(x_1,y_1),\dotsc,(x_n,y_n)\in\widetilde M$, and $\lambda_1,\dotsc,\lambda_n>0$ with $\sum_{i=1}^n \lambda_i =1$ such that $\mu=\sum_{i=1}^n \lambda_i m_{x_i,y_i}$. By Proposition~\ref{prop:delta_CM_characterisation}, 
$A = \{(x_i,y_i)\colon  i=1,\dotsc,n\}$ is cyclically monotonic. Thus, there exist $u,v\in M$ with $u\neq v$ such that $A\cup \{(u,v)\}$ and $A\cup \{(v,u)\}$ are $\gamma$-cyclically monotonic. By Proposition~\ref{prop:delta_CM_characterisation}, there exist $f,g\in B_{\Lip_0(M)}$ satisfying $f(m_{x,y})\geq \gamma$ and $g(m_{x,y})\geq \gamma$ for all $(x,y)\in A$, and $f(m_{u,v})\geq \gamma$ and $g(m_{v,u})\geq \gamma$. Therefore, we have $f(\mu)\geq \gamma$ and $g(\mu)\geq \gamma$, which implies $f,g\in S(\mu, \alpha)$, and
\[
\|f-g\|\geq (f-g)(m_{u,v})\geq 2\gamma.
\]
\end{proof}

From Lemma~\ref{lem: A cup (u,v) delta-CM} and Proposition~\ref{prop: w^*-LD2P characterisation with CM}, we immediately get another characterisation of the $w^*$-LD$2$P for the spaces of Lipschitz functions. 
\begin{definition}
    We say that $M$ has the \emph{$2$-Lip-LTP} if, given a finite cyclically monotonic subset $A$ of $\widetilde{M}$ and $\varepsilon>0$, there exist $f,g\in B_{\Lip_0(M)}$ and $u,v\in M$ with $u\neq v$ satisfying for all $(x,y)\in A$, $f(m_{x,y})\geq 1-\varepsilon$ and $g(m_{x,y})\geq 1-\varepsilon$, and for all $x,y\in \pi(A)$,
    \[
        \max\big\{f(x)-f(y),g(y)-g(x)\big\}+(1-\varepsilon) d(u,v)\leq d(x,u)+d(y,v).
    \]
\end{definition}

\begin{proposition}\label{prop: w^*-LD2P characterisation 2-Lip-LTP}
   The space $\Lip_0(M)$ has the $w^*$-LD$2$P if and only if $M$ has the $2$-Lip-LTP.
\end{proposition}

\section{\texorpdfstring{Separating LD$2$P and D$2$P in spaces of Lipschitz functions}{Separating LD2P and D2P in spaces of Lipschitz functions}}

In this section, we give an example of a metric space $M$ for which the space $\Lip_0(M)$ has the LD$2$P but not the D$2$P or even the $w^*$-D$2$P. To our knowledge, the only known examples of Banach spaces with the LD$2$P but without the D$2$P have appeared in \cite{MR4477423, MR3334951, MR4735972}. 

In \cite{MR3334951}, it was shown that every Banach space containing an isomorphic copy of $c_0$ admits a renorming with the LD$2$P but without the D$2$P. Similar ideas were recently applied in \cite{MR4735972} to special renormings of $L_\infty[0,1]$. It remains unclear whether any of these renormings is a dual Banach space. 

In \cite{MR4477423}, the authors constructed a Banach space with the LD$2$P but without the D$2$P, using a binary tree with the root removed. This space is not a dual space as it is separable and contains (an isomorphic copy of) $c_0$.

\begin{lemma}\label{lemma: natural number values function}
    Let $M$ be a pointed metric space with $d(x,y)\in \mathbb{N}\cup \{0\}$ for all $x,y\in M$, and let $A$ be a cyclically monotonic subset of $\widetilde{M}$. Then there exists $f\in B_{\Lip_0(M)}$ such that for all $x\in M$, $f(x)\in \mathbb{Z}$, and for all $(x,y)\in A$, $f(m_{x,y})=1$. 
\end{lemma}
\begin{proof}

    By Proposition~\ref{prop:delta_CM_characterisation}, there exists $g\in B_{\Lip_0(M)}$ satisfying $g(m_{x,y})=1$ for all $(x,y)\in A$. Define a function $f\colon M\to \mathbb{R}$ by $f(x)=\lfloor g(x)\rfloor$, the floor of $g(x)$. Note that $f \in B_{\Lip_0(M)}$ because for all $x,y\in M$,
    \begin{align*}
    f(x)-f(y)&=\lfloor g(x)\rfloor - \lfloor g(y)\rfloor \leq \lfloor g(y)+d(x,y)\rfloor - \lfloor g(y)\rfloor \\
    &= \lfloor g(y)\rfloor+d(x,y) - \lfloor g(y)\rfloor =d(x,y).
    \end{align*}
    Furthermore, note that if $(x,y)\in A$, then
    \begin{align*}
    f(x)-f(y) &= \lfloor g(x)\rfloor - \lfloor g(y)\rfloor = \lfloor g(y)+d(x,y)\rfloor - \lfloor g(y)\rfloor \\
    &= \lfloor g(y)\rfloor+d(x,y) - \lfloor g(y)\rfloor=d(x,y).        
    \end{align*}
\end{proof}

\begin{example}\label{Ex:new}
Let $M = \{x_i,y_i,u_i^j,v_i^j\colon i\in\{1,2,3\},j\in\N\}$ be a metric space where
\[
    d(y_1,x_2)=d(y_2,x_3)=d(y_3,x_1)=1,
\]
for all $i\in\{1,2,3\}$ and $j\in \N$,
\[
    d(x_i,u_i^j)=d(u_i^j,v_i^j)=d(v_i^j,y_i)=1,
\]
and the distance between two different elements is equal to $2$ in all other cases (see Figure~\ref{Figure_1}). Then the space $\Lip_0(M)$ lacks the $w^*$-D$2$P but has the LD$2$P.

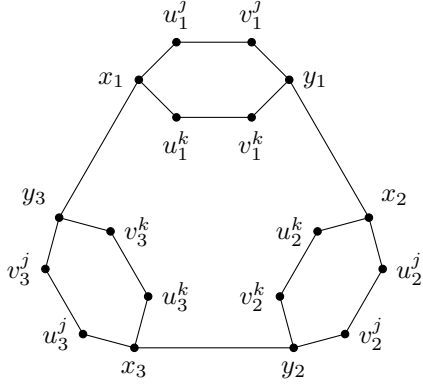
\begin{figure}[ht]
\begin{tikzpicture}
   \newdimen\R 
   \newdimen\r
   \R=1.8cm 
   \r=0.5cm 
   \coordinate (c) at (0,0);
   \foreach \a in {0,120,240}
   {
   \draw[black, rotate around={{\a+90}:(c)}] 
   (\R,2*\r) -- (-{\R/2+\r*sqrt(3)},{\R*sqrt(3)/2+\r});
   }
    
    \foreach \a/\uj/\x/\uk/\vk/\y/\vj/\puj/\px/\puk/\pvk/\py/\pvj in
    { 
        0/{$u_1^j$}/{$x_1$}/{$u_1^k$}/{$v_1^k$}/{$y_1$}/{$v_1^j$}/above/left/below/below/right/above,
        120/{$u_3^j$}/{$x_3$}/{$u_3^k$}/{$v_3^k$}/{$y_3$}/{$v_3^j$}/left/below/right/right/above left/left,
        240/{$u_2^j$}/{$x_2$}/{$u_2^k$}/{$v_2^k$}/{$y_2$}/{$v_2^j$}/right/above right/left/left/below/right
    }
    {
    \draw[black, rotate around={\a+90:(c)}] 
    (\R+\r, \r) node[inner sep=1pt, circle, draw, fill, label={\puj:\uj}] {}
    -- (\R, 2*\r) node[inner sep=1pt, circle, draw, fill, label={\px:\x}] {}
    -- (\R-\r, \r) node[inner sep=1pt, circle, draw, fill, label={\puk:\uk}] {}
    -- (\R-\r, -\r) node[inner sep=1pt, circle, draw, fill, label={\pvk:\vk}] {}
    -- (\R, -2*\r) node[inner sep=1pt, circle, draw, fill, label={\py:\y}] {}
    -- (\R+\r, -\r) node[inner sep=1pt, circle, draw, fill, label={\pvj:\vj}] {}
    -- cycle;
    }
\end{tikzpicture}

\caption{A representation of the metric space M in Example~\ref{Ex:new}. The
distances between points connected by a straight line segment are 1, the
distances between other different points are 2.
}\label{Figure_1}
\end{figure}

\end{example}

\begin{proof}
We start by showing that the space $\Lip_0(M)$ does not have the $w^*$-D$2$P. By Proposition~\ref{prop: w^*-D2P characterisation}, it suffices to show that the space $M$ lacks the Lip-LTP. Without loss of generality, we assume that $0=x_1$. Let $N=\{x_1,x_2,x_3, y_1,y_2,y_3\}$ and let $\varepsilon=1/14$. Define $f\in S_{\Lip_0(M)}$ by
\[
f(x)=\begin{cases}
    0,& \text{if $x\in \{x_1,y_3\}$}; \\
    1/2, & \text{if $x=y_2$};\\
    3/2, \quad &\text{if $x\in \{y_1,x_3\}$};\\
    2, &\text{if $x=x_2$};\\
    1, &\text{otherwise}.
\end{cases}
\]
Fix $u,v\in M$ with $u\neq v$. We show that there exist $x,y\in N$ such that the inequality~\eqref{ineq: Lip-LTP} does not hold. First, assume that $u\in N$ or $v\in N$. Without loss of generality, let $u\in N$. Let $x=u$ and let $y\in N$ be such that $|f(x)-f(y)|=3/2$. Then 
\[
    (1-\varepsilon)\big(|f(x)-f(y)|+d(u,v)\big)\geq \frac{13}{14}\cdot \frac{5}{2}> 2 \geq d(x,u)+d(y,v).
\]
Now assume that $u,v\in M\setminus N$. If $d(u,v)=2$, then there exist $x,y\in N$ satisfying $d(x,u)=1$ and $|f(x)-f(y)|=3/2$. Then
\[
    (1-\varepsilon)\big(|f(x)-f(y)|+d(u,v)\big)=\frac{13}{14}\cdot \frac{7}{2}>3\geq d(x,u)+d(y,v).
\]
If $d(u,v)=1$, then there exist $i\in\{1,2,3\}$ and $j\in \N$ such that $u,v\in \{u_i^j,v_i^j\}$. Therefore, there exist $x,y\in \{x_i,y_i\}$ satisfying $d(x,u)=1$, $d(y,v)=1$, and $|f(x)-f(y)|=3/2$. Thus,
\[
    (1-\varepsilon)\big(|f(x)-f(y)|+d(u,v)\big)=\frac{13}{14}\cdot \frac{5}{2}>2=d(x,u)+d(y,v).
\]
Consequently, $\Lip_0(M)$ does not have the $w^*$-D$2$P.

Next, we show that the space $\Lip_0(M)$ has the LD2P. We make use of Proposition~\ref{prop: LD2P characterisation}. Let $\mu\in S_\ba$ be optimal and let $\gamma\in (0,1)$. By Proposition~\ref{prop: optimal mu Z}, there exists a cyclically monotonic subset $B$ of $\widetilde{M}$ with $\mu(B)\geq \frac{1+\gamma}{2}$. Choose $j\in \mathbb{N}$ such that $C=\pi^{-1}(\{u_1^j,u_2^j,u_3^j,v_1^j,v_2^j,v_3^j\})$ satisfies $\mu(C)\leq \frac{1-\gamma}{2}$. Set $A=B\setminus C$ and note that $\mu(A)\geq \gamma$. 

By Lemma~\ref{lemma: natural number values function}, there exists $f\in B_{\Lip_0(M)}$ such that for all $x\in M$, $f(x)\in \mathbb{Z}$, and for all $(x,y)\in A$, $f(m_{x,y})=1$. Without loss of generality, we may and do assume that $0\in \pi(A)$ and that $f(0)=\min_{x\in \pi(A)}{f(x)}$. Then $f(x)\in \{0,1,2\}$ for all $x\in \pi(A)$. 

It remains to show that there exist $u,v\in M$ with $u\neq v$ satisfying for all $x,y\in \pi(A)$,
\begin{equation*}\label{eq: ex_LD2P_|f|}
|f(x)-f(y)|+d(u,v)\leq d(x,u)+d(y,v).
\end{equation*}
Indeed, if this is the case, then set $g=f$ and note that the subset $A$ of $\widetilde{M}$, functionals $f,g\in B_{\Lip_0(M)}$, and elements $u,v\in M$ satisfy the conditions of Proposition~\ref{prop: LD2P characterisation}.

If there exists $i \in \{1, 2, 3\}$ such that $ x_i \not\in \pi(A) $ or $ y_i \not\in \pi(A) $, then letting $ u = u_i^j $ and $ v = v_i^j $, it follows that for all $ x, y \in \pi(A) $, $ d(x, u) \geq 2 $ or $ d(y, v) \geq 2 $, and hence 
\[
|f(x) - f(y)| + d(u, v) \leq d(x, y) + 1 \leq 3 \leq d(x, u) + d(y, v).
\]

Now assume that $x_1,x_2,x_3,y_1,y_2,y_3\in \pi(A)$. Note that $|f(x_i)-f(y_i)|\leq 1$ for some $i\in \{1,2,3\}$. Indeed, suppose for example that $f(x_1)-f(y_1)=2$. Then $f(x_1)=2$ and $f(y_1)=0$, and therefore, $f(x_2)\leq 1$ and $f(y_3)\geq 1$. Thus, if $|f(x_2)-f(y_2)|>1$ and $|f(x_3)-f(y_3)|>1$, then $f(x_2)=0$, $f(y_2)=2$, $f(y_3)=2$, and $f(x_3)=0$. But then $\|f\|\geq  2$ because
\[
f(y_2)-f(x_3)=2=2d(y_2,x_3).
\]
Let $i\in \{1,2,3\}$ be such that $|f(x_i)-f(y_i)|\leq 1$. Let $u=u_i^j$ and $v=v_i^j$, and fix $x,y\in \pi(A)$. Note that $d(u,v)=1$, $d(x,u)\geq 1$, and $d(y,v)\geq 1$. If $x\not\in\{x_i,y_i\}$ or $y\not\in \{x_i,y_i\}$, then $d(x,u)\geq 2$ or $d(y,v)\geq 2$, and therefore,
\[
|f(x)-f(y)|+d(u,v)\leq d(x,y)+1\leq 3 \leq d(x,u)+d(y,v).
\]
If $x,y\in \{x_i,y_i\}$, then $|f(x)-f(y)|\leq 1$, and therefore,
\[
|f(x)-f(y)|+d(u,v)\leq 2\leq d(x,u)+d(y,v).
\]

\end{proof}

\section*{Acknowledgements}
This work was supported by the Estonian Research
Council grant (PRG1901).

The authors thank Trond A.~Abrahamsen and Vegard Lima for helpful discussions on the content of this paper.

\bibliographystyle{amsplain}
\bibliography{references}
\end{document}